\documentclass[11pt]{amsart}

\usepackage{amsmath}
\usepackage{amssymb}
\usepackage{amsfonts}
\usepackage{latexsym}
\usepackage{color}
\usepackage{bm}

\newtheorem{theo}{Theorem}[section]
\newtheorem{thm}[theo]{Theorem}
\newtheorem{lem}[theo]{Lemma}
\newtheorem{prop}[theo]{Proposition}
\newtheorem{cor}[theo]{Corollary}
\newtheorem{lemma}[theo]{Lemma}

 \theoremstyle{definition}
\newtheorem{definition}[theo]{Definition}

\newtheorem{example}[theo]{Example}
 \theoremstyle{remark}

 \numberwithin{equation}{section}

\newtheorem{remark}[theo]{Remark}

\newcommand{\betheo}{\begin{theo}$\!\!\!${\bf } }
\newcommand{\entheo}{\end{theo}}

\newcommand{\becor}{\begin{cor}$\!\!\!$  }
\newcommand{\encor}{\end{cor}}

\newcommand{\belem}{\begin{lem}$\!\!\!${\bf .} }
\newcommand{\enlem}{\end{lem}}

\newcommand{\beprop}{\begin{prop}$\!\!\!${\bf .} }
\newcommand{\enprop}{\end{prop}}

\newcommand{\bedefi}{\begin{definition}$\!\!\!$ \rm }
\newcommand{\findefi}{ \end{definition}}

\newcommand{\beex}{\begin{example}$\!\!\!$ \rm }
\newcommand{\enex}{ \end{example}}

\newcommand{\berem}{\begin{remark}$\!\!\!$ \rm }
\newcommand{\enrem}{ \end{remark}}

\newcommand{\be}{\begin{equation}}
\newcommand{\en}{\end{equation}}

\newcommand{\bea}{\begin{eqnarray}}
\newcommand{\ena}{\end{eqnarray}}

\newcommand{\beano}{\begin{eqnarray*}}
\newcommand{\enano}{\end{eqnarray*}}

\newcommand{\bee}{\begin{enumerate}}
\newcommand{\ene}{\end{enumerate}}

\newcommand{\bei}{\begin{itemize}}
\newcommand{\eni}{\end{itemize}}

\newcommand{\betab}{\begin{tabular}}
\newcommand{\entab}{\end{tabular}}

\newcommand{\bd}{\begin{displaymath}}

\def\D{{\mathcal D}}

\def\H{{\mathcal H}}

\def\K{{\mathcal K}}

\def\M{{\mathcal M}}

\def\J{\relax\ifmmode {\mathcal J}\else${\mathcal J}$\fi}
\def\x{\relax\ifmmode {\mbox{*}}\else*\fi}

\newcommand{\mc}{\mathcal}
\newcommand{\mb}{\mathbb}

\newcommand{\vp}{\varphi}

\newcommand{\ip}[2]{\left\langle {#1}\left|{#2}\right.\right\rangle}

\def\OL{\relax\ifmmode {\sf L}\else{\textsf L}\fi}
\def\OR{\relax\ifmmode {\sf R}\else{\textsf R}\fi}

\newcommand{\RN}{\mb R}

\newcommand{\bdim}{\begin{proof}}
  \newcommand{\edim}{\end{proof}}

\begin{document}

\title[Non-self-adjoint resolutions of the identity]{Non-self-adjoint resolutions of the identity and associated operators\\}

\author[A. Inoue]{Atsushi Inoue}
\address{Department of Applied Mathematics, Fukuoka University, Fukuoka 814-0180, Japan}
\email{a-inoue@fukuoka-u.ac.jp}
\author[C.Trapani]{Camillo Trapani}
\address{%
Dipartimento di Matematica e Informatica \\
Universit\`a di Palermo\\
I-90123 Palermo\\
Italy}
\email{camillo.trapani@unipa.it}

\maketitle

\begin{abstract} Closed operators in Hilbert space defined by a non-self-adjoint resolution of the identity $\{X(\lambda)\}_{\lambda\in {\mb R}}$, whose adjoints constitute also a resolution of the identity, are studied . In particular, it is shown that a closed operator $B$ has a spectral representation analogous to the familiar one for self-adjoint operators if and only if $B=TAT^{-1}$ where $A$ is self-adjoint and $T$ is a bounded operator with bounded inverse.
\end{abstract}
\section{Introduction}
In recent years there has been an increasing interest on non-self-adjoint operators with real spectrum,  because of the important role they play in the so-called pseudo-hermitian quantum mechanics,  an unconventional approach to this branch of physics, based on the use of non-self-adjoint Hamiltonians \cite{bender, mosta1}. Often self-adjointness can be restored by changing the {environment}: in fact, if a closed operator $H'$ can be expressed as $H'=THT^{-1}$, where $H$ is self-adjoint and $T$ is a bounded operator with bounded inverse then $H$ and $H'$ have the same spectrum. If this condition is satisfied, $H'$ and $H$ are said to be {\em similar}.
This relation can be interpreted as the possibility of defining a (possibly, indefinite) inner product which makes $H'$ into a self-adjoint operator  with respect to the new metric.
In this case, if $H= \int_{\mb R} \lambda dE(\lambda)$ is the spectral representation of $H$, then $H'= \int_{\mb R} \lambda dX(\lambda)$ where $X(\lambda)=TE(\lambda)T^{-1}$. The family $\{X(\lambda)\}_{\lambda\in {\mb R}}$ obtained in this way behaves under many respects in analogous way to an ordinary spectral family with the crucial difference that its elements are non-self-adjoint projections (this means that, for every $\lambda \in {\mb R}$, $X(\lambda)^2=X(\lambda)$ but $X(\lambda)^*\neq X(\lambda)$, in general).

This observation is the starting point of this paper. In fact, we will consider, as suggested by the previous example, a non-self-adjoint resolution of the identity $\{X(\lambda)\}_{\lambda\in {\mb R}}$  enjoying prescribed regularity properties (monotonicity,  uniform boundedness, etc.); in particular we will focus our attention to the case where $\{X(\lambda)^*\}_{\lambda\in {\mb R}}$ is a resolution of the identity too (we speak in this case of a *-resolution of the identity) and study closed operators that are associated to it. To be clearer, let us consider an {\em ordinary} spectral family $\{E(\lambda)\}_{\lambda\in {\mb R}}$ consisting of self-adjoint (or orthogonal) projections in Hilbert space $\H$. Then, as it is well-known, this family defines uniquely a self-adjoint operator $A$ whose domain
$$D(A)=\left\{ \xi \in \H: \int_{\mb R}\lambda^2 d\ip{E(\lambda)\xi}{\xi} <\infty\right\} $$ can be expressed in several equivalent ways due to the equalities
\begin{equation}\label{eq_equalities} \ip{E(\lambda)\xi}{\xi}= \|E(\lambda)\xi\|^2 = \ip{(E(\lambda)^*E(\lambda))^{1/2}\xi}{\xi}, \quad \xi \in \H.\end{equation}
These equalities do not hold, in general, if we remove the assumption that each $E(\lambda)$ is self-adjoint, so that the corresponding spectral integrals may produce different operators and these operators are the main object of this paper.
The main result of the paper
consists in showing that a closed operator $B$ can be expressed as
 $$ B=\int_{\mb R} \lambda dX(\lambda)$$ if, and only if it is similar to a self-adjoint operator $A$; i.e. $B=TAT^{-1}$ where $A$ is self-adjoint and $T$ is a bounded operator with bounded inverse.

\section{Preliminaries}
In this section we collect some definitions and facts concerning {\em quasi-similarity} and {\em similarity} for possibly unbounded linear operators \cite{jpa_ct_pipandmetric}.

\medskip
Let $\H, \K$ be Hilbert spaces, $D(A)$ and $D(B)$ dense subspaces, 
 respectively of $\H$ and $\K$;   $A:D(A) \to \H$, $B: D(B) \to \K$ two linear operators.
 A bounded operator $T:\H \to \K$ is called an {\em intertwining} operator for $A$ and $B$ if
\begin{itemize}
\item[(i)] $T:D(A)\to D(B)$;
\item[(ii)] $BT\xi = TA\xi, \; \forall \xi \in D(A)$.
\end{itemize}

\bedefi
Let  $A$ and $B$  be two linear operators in the Hilbert spaces $\H$ and $\K$, respectively.

We say that $A$ and $B$ are {\em quasi-similar}, and write $A\dashv B$,
if there exists an intertwining operator $T$ for $A$ and $B$ which is invertible, with inverse $T^{-1}$ densely defined.

The operators $A$ and $B$ are said to be {\em similar}, and write $A\sim B$,  if they are quasi similar and the inverse $T^{-1}$ of the intertwining operator $T$ intertwines $B$ and $A$.
\findefi

\berem We notice that  $\sim$ is an equivalence relation. Moroever, if $A\sim B$, then $TD(A)=D(B)$.  \end{remark}

If $A\dashv B$ (respectively, $A\sim B$) with intertwining operator $T$, then, $B^*\dashv A^*$ (respectively, $B^*\sim A^*$), with intertwining operator $T^*$.

\berem Let $A$ and $B$ be linear operators in $\H$ and $\K$, respectively, with $A\sim B$. The following properties of similar operators are easily proved.
\begin{itemize}
\item[(i)] $A$ is closed if, and only if, $B$ is closed.
\item[(ii)] $A^{-1}$ exists if, and only if,  $B^{-1}$ exists. Moreover, $B^{-1} \sim A^{-1}$.
\end{itemize}
\enrem

\medskip

If $A$ is a closed operator, we denote, as usual, by $\sigma(A)$ its spectrum. The parts in which the spectrum is traditionally decomposed, the point spectrum, the continuous spectrum and the residual spectrum, are denoted respectively by  $\sigma_p (A)$, $\sigma_c(A)$, $\sigma_r(A)$.

\begin{prop}\label{prop_2.4} Let $A$, $B$ be closed operators. Assume that $A\sim B$ and let $T$ be the corresponding intertwining operator. Then
the spectra $\sigma (A)$ and $\sigma(B)$ coincide and
$$\sigma_p (A)=\sigma_p(B), \quad \sigma_c (A)=\sigma_c(B) , \quad \sigma_r (A)=\sigma_r(B).$$
Moreover, if $\lambda \in \sigma_p (A)$, the multiplicity $m_A(\lambda)$ of $\lambda$ as eigenvalue of $A$ is the same of its multiplicity $m_B(\lambda)$
 as eigenvalue of $B$.
\end{prop}

The situation for quasi-similarity is more involved and it has been described in \cite{jpa_ct_pipandmetric}. We summarize in the next proposition the main results.
\begin{prop} Let $A$, $B$ be closed operators. Assume that $A\dashv B$ with intertwining operator $T$.
Then the following statements hold.
\begin{itemize}
\item[{\sf (sp.1)}]$\sigma_p(A)\subseteq \sigma_p(B)$ and for every $\lambda \in \sigma_p(A)$ one has $m_A(\lambda) \leq m_B(\lambda)$, where $m_C(\lambda)$
 denotes the multiplicity $\lambda$ as eigenvalue of the operator $C$.
\item[{\sf (sp.2)}]If  $T^{-1}$ is bounded and  $T(D(A))$ is a core for $B$, then $\sigma_p(B)\subseteq \sigma(A)$
\item[{\sf (sp.3)}] If $T^{-1}$ is everywhere defined and bounded and $TD(A)$ is a core for $B$. Then
$$ \sigma_p (A)\subseteq \sigma_p (B) \subseteq \sigma (B) \subseteq \sigma(A).$$
\end{itemize}
\end{prop}

\berem Suppose, for instance, that $A$ is self-adjoint, then any operator $B$ which is quasi-similar to $A$ by means on an intertwining operator $T$ whose inverse is bounded too, has real spectrum and, if $A$ has a pure point spectrum, then $B$ is isospectral to $A$.
\enrem

\section{Non-self-adjoint resolutions of the identity}
To begin with, we fix some terminology.
Let $\H$ be a Hilbert space.
A bounded operator $X$ will be called a projection if $X^2=X$ and a self-adjoint (or orthogonal) projection if $X=X^2=X^*$.
If $X$ is a nonzero projection, then $\|X\|\geq 1$, while if it is self-adjoint $\|X\|=1$.

\beex \label{ex_main}Let us consider two biorthogonal Schauder bases $\Phi=\{\varphi_n, \,n\in {\mb N}\}$ and $\Psi=\{\psi_n, \,n\in {\mb N}\}$ of the Hilbert space $\H$, $ \ip{\varphi_i}{\psi_j}=\delta_{i,j}$ and let us consider
an operator of the form
\begin{equation} \label{X_form}S=\sum_{k=1}^\infty \alpha_k (\psi_k\otimes \overline{ \varphi_k})\end{equation}
with $\alpha_k \in {\mb C}$, $k\in {\mb N}$.
The domain of $S$ is the following subspace of $\H$:
$$D(S)= \left\{ \xi \in \H: \lim_{n\to \infty} \left\| \sum_{k=n+1}^{n+p} \alpha_k \ip{\xi}{\varphi_k}\psi_k \right\|=0, \forall p\in {\mb N}  \right\}.$$
This domain is dense, since it contains the vectors $\psi_k$, $k \in {\mb N}$.

It is easy to see that every $\alpha_k$ is an eigenvalue of $S$ with eigenvector $\varphi_k$. The spectrum $\sigma(S)$ of $S$ is the set $\overline{\{\alpha_k,\, k\in {\mb N}\}}$. In particular, $\sigma_p(S)=\{\alpha_k,\, k\in {\mb N}\}$ and every limit point of $\sigma_p(S)$, if any, lies in the continuous spectrum $\sigma_c(S)$ of $S$.

To simplify notations, we put $R_k=\psi_k\otimes \overline{ \varphi_k}$. This family of rank one operators enjoys the following easy properties:
\begin{itemize}
\item[(i)] $\|R_k\|\leq \|\varphi_k\|\, \|\psi_k\|$;
\item[(ii)] $R_k^*=L_k:= \varphi_k \otimes \overline{\psi_k}$;
\item[(iii)] $R_k^2 =R_k$ and $R_k R_m=0$ if $m\neq k$.
\end{itemize}
In particular, (iii) implies that $R_k$ is a non-self-adjoint projection (unless $\varphi_k=\psi_k$). Moreover, since $\Psi=\{\psi_n, \,n\in {\mb N}\}$ is a Schauder basis, one gets
$$ \xi= \sum_{k=1}^\infty R_k \xi, \quad \forall \xi \in \H.$$

Thus, the family $\{R_k\}$ enjoys the property
$$ \sup_{n\in {\mb N}} \left\|\sum_{k=1}^nR_k\right\|<\infty.$$

Let us now assume that the spectrum  $\sigma(S)$ of $S$ is real. Then, we can define, for $\lambda \in {\mb R}$ and $\xi \in \H$,
$$ X(\lambda) \xi = \sum_{k\leq \lambda} R_k\xi . $$
Then we can formally write
$$ S \xi= \int_{\mb R} \lambda dX(\lambda) \xi .$$

 Let us suppose that $\{\varphi_k\}$ and $\{\psi_k\}$ are biorthogonal Riesz bases. This means that there exists a symmetric bounded operator $G$ with bounded inverse $G^{-1}$ and an orthonormal basis $\{\chi_n\}$ such that $\varphi_k = G^{-1}\chi_k$ and $\psi_k=G\chi_k$, for every $k \in {\mb N}$.
Then, we get
$$(\psi_k \otimes \overline{\vp_k})\xi = \ip{\xi}{\vp_k}\psi_k =  \ip{\xi}{G^{-1}\chi_k}G\chi_k = \ip{G^{-1}\xi}{\chi_k}G\chi_k.$$
Hence
$\psi_k \otimes \overline{\vp_k}= G(\chi_k \otimes \overline{\chi_k})G^{-1}$.

Then, it is easily seen that the family of operators $\{X(\lambda)\}_{\lambda\in {\mb R}}$ enjoys the properties ({\sf qs$_1$})-({\sf qs$_4$}) listed in Definition \ref{resofid} below.

We remark that in finite dimensional spaces every family of projections whose sum is the identity operator is similar to a family of orthogonal projections; so that the situation discussed above is the more general possible. For the infinite dimensional case, an analogous statement was obtained by Mackey \cite[Theorem 55]{mackey}: every non-self-adjoint resolution of the identity is similar to a self-adjoint resolution of the identity (Mackey's terminology is different: a non-self-adjoint resolution of the identity is a countably additive spectral measure on the Borel sets of the plane or of the real line); the resolution of the identity $\{X(\lambda)\}$ of the next Definition \ref{resofid} need not define a countably additive spectral measure on the Borel sets.

\enex

\bedefi \label{resofid} Let $\H$ be a Hilbert space. A {\em resolution of the identity} of $\H$ on the interval $I:=[\alpha, \beta]$, $(-\infty\leq \alpha < \beta\leq +\infty)$ is a one parameter family of (non necessarily self-adjoint) bounded operators $\{X(\lambda)\}_{\lambda \in I}$ satisfying
the following conditions
\begin{itemize}
\item[(\sf qs$_1$)] $\displaystyle \sup_{\lambda \in I} \|X(\lambda)\|:=\gamma(X)<+\infty$;
\item[(\sf qs$_2$)] $X(\lambda)X(\mu)=X(\mu)X(\lambda)= X(\lambda)$ if $\lambda < \mu$;
\item[(\sf qs$_3$)] $\displaystyle \lim_{\lambda \to \alpha}X(\lambda)\xi = 0; \quad \lim_{\lambda \to \beta}X(\lambda)\xi = \xi, \quad \forall\xi \in \H$;
\item[(\sf qs$_4$)] $\displaystyle \lim_{\epsilon \to 0^+} X(\lambda +\epsilon)\xi= X(\lambda)\xi, \; \forall \lambda \in I;\; \forall\xi \in \H$.
\end{itemize}
If the limits in ({\sf qs$_3$}) and ({\sf qs$_4$}) hold with respect to the weak topology only, then we say that $\{X(\lambda)\}_{\lambda \in I}$ is a {\em weak resolution of the identity}.

If $X(\lambda)^*=X(\lambda)$, for every $\lambda \in I$, we say that the resolution of the identity is {\em self-adjoint}.
\findefi
\berem  Since the $X(\lambda)$'s are projections, $\|X(\lambda)\|\geq 1$, for every $\lambda \in I$. Hence $\gamma(X)\geq 1$. \enrem 
\begin{prop} If $\{X(\lambda)\}_{\lambda \in {\mb R}}$ is a resolution of the identity, $\{X(\lambda)^*\}_{\lambda \in {\mb R}}$ is a weak resolution of the identity.\end{prop}

\berem From ({\sf qs$_2$}) and ({\sf qs$_4$}) it follows that $X(\lambda)^2=X(\lambda)$,  for every $\lambda \in {\mb R}$. Thus every $X(\lambda)$ is a projection, but not an {\em orthogonal} projection, in general. If also $X(\lambda)^*=X(\lambda)$, for every $\lambda \in {\mb R}$, then $\{X(\lambda)\}_{\lambda \in {\mb R}}$ is a spectral family in the usual sense. In this case, weak limits, automatically become strong limits so every weak self-adjoint resolution of the identity is a resolution of the identity.
\enrem
\berem If $\{X(\lambda)\}_{\lambda \in {\mb R}}$ is a self-adjoint resolution of the identity, then condition {(\sf qs$_2$)} can be replaced with the following equivalent one:
\begin{itemize}
\item[{(\sf qs$_2$')}] $X(\lambda)\leq X(\mu)$, $\lambda, \mu \in {I}, \,\lambda <\mu$.
\end{itemize}
\enrem

\bedefi If both  $\{X(\lambda)\}_{\lambda \in {I}}$ and $\{X(\lambda)^*\}_{\lambda \in {I}}$ are resolutions of the identity, we simply say that $\{X(\lambda)\}_{\lambda \in {I}}$ is a {\em *-resolution of the identity}.
\findefi

For reader's convenience, we recall the definition of generalized resolution of the identity due to Naimark, \cite[Appendix]{riesz}, \cite[Vol.II, Appendix]{Akhiezer}.

\bedefi \label{def_naimark} A {\em generalized resolution of the identity} is a one parameter family of bounded symmetric  operators $\{B(\lambda)\}_{\lambda \in I}$, where $I:=[\alpha, \beta]$ is a bounded or unbounded interval of the real line, satisfying
the following conditions
\begin{itemize}
\item[\sf(gri$_1$)] $\lim_{\lambda\to \alpha}B(\lambda)\xi=0$, \; $\lim_{\lambda\to \beta}B(\lambda)\xi=\xi$, $\;\forall \xi \in \H$;
\item[\sf(gri$_2$)] $\lim_{\epsilon \to 0^+} B(\lambda +\epsilon)\xi= B(\lambda)\xi, \; \forall \lambda \in {I}$, $\;\forall \xi \in \H$;
\item[\sf(gri$_3$)] $B(\lambda) \leq B(\mu)$, if $\lambda, \mu \in I$, $\lambda<\mu$.
\end{itemize}
\findefi

\berem
The operators of a generalized resolution of the identity are self-adjoint but not necessarily idempotent, while in Definition \ref{resofid} we require exactly the opposite. \enrem

\beex \label{ex_310}Let $\{E(\lambda)\}_{\lambda \in {\mb R}}$ be a self-adjoint resolution of the identity and $G$ a bounded symmetric operator with bounded inverse, with $G^2\neq I$. Put $X(\lambda):= GE(\lambda)G^{-1}$,  for every $\lambda \in {\mb R}$. Then $\{X(\lambda)\}_{\lambda \in {\mb R}}$ is a resolution of the identity in the sense of Definition \ref{resofid}. In particular, $1\leq \|X(\lambda)\|\leq \|G\|\|G^{-1}\|$, for every $\lambda \in {\mb R}$. Let us prove, for instance, the second equality in ({\sf qs$_3$}). We have,
\begin{align*} &\|GE(\lambda)G^{-1}\xi -\xi\|= \|GE(\lambda)G^{-1}\xi -GG^{-1}\xi\|\\
&\leq \|G\| E(\lambda)G^{-1}\xi- G^{-1}\xi\|\to 0\; \mbox{as} \; \lambda \to +\infty.\end{align*}
It is easily seen that in this case $\{X(\lambda)^*\}_{\lambda \in {\mb R}}$ is a (non-self-adjoint) resolution of the identity too.

If $G^2=I$ then every $X(\lambda)$ is symmetric and so it is a self-adjoint resolution of the identity.
\enex

\medskip
From now on, we will only consider the case $$\gamma(X)= \sup_{\lambda \in I} \|X(\lambda)\| =1.$$
This assumption does not imply that the $X(\lambda)$'s are self-adjoint projections. For instance, in the case considered in Example \ref{ex_310}, one can easily find examples of operators $G$ satisfying  $\|G\|\|G^{-1}\|=1$.

\begin{lemma}\label{lemma_310}Let $\{X(\lambda)\}_{\lambda \in {\mb R}}$ be a resolution of the identity. If $\lambda <\mu$ then \begin{equation}\label{ineq_1}\|X(\lambda)\xi\|\leq  \|X(\mu)\xi\|, \quad \forall\xi \in \H.\end{equation} Hence, the nonnegative valued function $\lambda \mapsto \|X(\lambda)\xi\|$ is increasing, for every $\xi\in \H$. \end{lemma}
\begin{proof}
Indeed, if $\lambda\leq \mu$,
$$ \|X(\lambda)\xi\|= \|X(\lambda)X(\mu)\xi\|\leq \|X(\lambda)\| \, \|X(\mu)\xi\|\leq \gamma(X)\|X(\mu)\xi\|=\|X(\mu)\xi\|.$$
\end{proof}

\subsection{Operators associated to a resolution of the identity}

A resolution of the identity $\{X(\lambda)\}$ defines
an operator valued function $\lambda \mapsto F(\lambda)$,  $\lambda \in {\mb R}$, where $F(\lambda)$ is the positive operator
\begin{equation}\label{familyF} F(\lambda)=X(\lambda)^* X(\lambda), \quad \lambda \in {\mb R}.\end{equation}
Of course, $\ip{F(\lambda)\xi}{\xi}=\|X(\lambda)\xi \|^2$ for every $\lambda \in {\mb R}$, $\xi \in \H$.

\begin{lemma}\label{lemma_sfunction} Let $\{X(\lambda)\}$ be a  *-resolution of the identity. Then, the operator valued function $\lambda \to F(\lambda)$ has the following properties:
\begin{itemize}
\item[(\sf fs$_1$)] $\displaystyle \sup_{\lambda \in {\mb R}} \|F(\lambda)\|={1}$;
\item[(\sf fs$_2$)] $F(\lambda)\leq F(\mu)$ if $\lambda < \mu$;
\item[(\sf fs$_3$)] $\displaystyle \lim_{\lambda \to -\infty}F(\lambda)\xi = 0; \quad \lim_{\lambda \to +\infty}F(\lambda)\xi = \xi, \quad \forall \xi \in \H$;
\item[(\sf fs$_4$)] $\displaystyle \lim_{\epsilon \to 0^+} F(\lambda +\epsilon)\xi= F(\lambda)\xi, \; \forall \lambda \in {\mb R};\; \forall\xi \in \H$.
\end{itemize}
Hence, $\{F(\lambda)\}_{\lambda \in {\mb R}}$ is a generalized resolution of the identity in the sense of Definition \ref{def_naimark}.
\end{lemma}
\begin{proof} ({\sf fs}$_1$) is an easy consequence of the C*-property.

({\sf fs}$_2$): Using Lemma \ref{lemma_310}  we have $$\ip{F(\lambda)\xi}{\xi}=\|X(\lambda)\xi \|^2 \leq \|X(\mu)\xi\|^2 = \ip{F(\mu)\xi}{\xi}.$$

({\sf fs}$_3$) follows from the inequalities
$$ 0\leq \|F(\lambda)\xi\| = \|X(\lambda)^* X(\lambda)\xi\|\leq \|X(\lambda)^*\| \|X(\lambda)\xi\|\to 0 \; \mbox{as} \; \lambda \to -\infty$$
and
\begin{align*}
&\|F(\lambda)\xi- \xi\|= \|X(\lambda)^* X(\lambda)\xi - \xi\| \\
&=  \|X(\lambda)^* X(\lambda)\xi -X(\lambda)^*\xi + X(\lambda)^*\xi - X(\lambda)\xi + X(\lambda)\xi-\xi\| \\
&\leq \|X(\lambda)^*\| \|X(\lambda)\xi -\xi\|+ \|X(\lambda)^*\xi-X(\lambda)\xi\| + \|X(\lambda)\xi -\xi\| \to 0 \,\mbox{as} \, \lambda \to +\infty
\end{align*}
since, by the assumption $\lim_{\lambda \to +\infty}X(\lambda)^*\xi=\lim_{\lambda \to +\infty}X(\lambda)\xi = \xi,$ for every $\xi \in \H$.

({\sf fs}$_4$): In similar way, since
\begin{align*}
&\|F(\lambda +\epsilon)\xi - F(\lambda)\xi\| =\|X (\lambda +\epsilon)^*X(\lambda+\epsilon) \xi - X(\lambda)^*X(\lambda)\xi\|\\
&\leq \|X (\lambda +\epsilon)\|\|X (\lambda +\epsilon) - X(\lambda)\xi\| +\|X (\lambda +\epsilon)^*X(\lambda) \xi - X(\lambda)^*X(\lambda)\xi\|,
\end{align*}
both terms go to $0$ by the assumption as $\epsilon \to 0^+$.
\end{proof}

So under the assumptions of Lemma \ref{lemma_sfunction} one can define, in standard fashion, a positive operator valued measure on the Borel sets of the line: one begins with considering
a bounded interval of the form
$\Delta=]\lambda, \mu]$ and defines
$$F(\Delta)= F(\mu)- F(\lambda).$$
the measure of the closed interval $[\lambda, \mu]$ is then defined by
$F([\lambda, \mu])= F(\{\lambda\})+F(\Delta)$,  the measure of the singleton $\{\lambda\}$ being defined as
$$F(\{\lambda\})= \lim_{\epsilon\to 0^+} F(]\lambda-\epsilon,
\lambda])$$
and then one extends the measure $F(\cdot)$ to arbitrary Borel sets.
\begin{lemma} \label{lemma_f} The following properties hold:
\begin{itemize}
\item[({\sf f}$_1$)]$ X(\mu)^*F(\lambda) X(\mu)= F(\lambda)$, if $\lambda\leq \mu$;
\item[({\sf f}$_2$)]$F(\lambda)=X(\mu)^*F(\lambda),\quad F(\lambda)=F(\lambda)X(\mu)$, if $\lambda\leq \mu$;
\item[({\sf f}$_2$)] If $a<b$,
\begin{align*}(X(b)^* - X(a)^*) &F(\lambda) (X(b)-X(a)) \\&=\left\{ \begin{array}{ll} 0 & \mbox{ if } \lambda <a< b\\
(X(\lambda)^* - X(a)^*) (X(\lambda)-X(a))  & \mbox{ if }  a<\lambda <b \\
(X(b)^* - X(a)^*)(X(b)-X(a)) & \mbox{ if }  a<b<\lambda .
\end{array}    \right. \end{align*}

\end{itemize}
\end{lemma}
These properties come almost immediately from the definition and from ({\sf qs}$_1$).

\medskip
If $\{X(\lambda)\}$ is a *-resolution of the identity, then the family $\{F(\lambda)\}$ defined above is a generalized spectral family in the sense of Naimark \cite[Appendix]{riesz}. Hence, there exists a self-adjoint resolution of the identity $\{\mathbf{E}(\lambda)\}$ in a possibly larger Hilbert space $\bm{\H}$, containing $\H$ as a closed subspace, such that $F(\lambda)= \mathbf{PE}(\lambda)\upharpoonright \H$, where $\mathbf{P}$ is the projection of  $\bm{\H}$ onto $\H$ and by requiring that $\bm{\H}$ is spanned by the vectors of the form $\mathbf{E}(\lambda)\xi$, $\xi \in \H$, then $\bm{\H}$ is (essentially) unique.

Let $\bm{A}$ be the self-adjoint operator, with dense domain $\bm{PD(A)}$ in $\bm{\H}$ whose spectral family is $\{\mathbf{E}(\lambda)\}$. Then the operator $T_X = \bm{PA}\upharpoonright \H$ on the dense domain $D(T_X)=\bm{PD(A)}$ of $\H$ is closed and symmetric.

Then one easily proves that

\begin{thm}\label{thm_opertospectfam} Assume that $\{X(\lambda)\}$ is a *-resolution of the identity and let $F(\cdot)= X(\cdot)^*X(\cdot)$ be the positive operator valued function defined above.

Set $$D(T_X)=\left\{\xi \in \H:\, \int_\RN
\lambda ^2 \,d\ip{F(\lambda)\xi}{\xi} < \infty\right\}.$$ Then, $D(T_X)$ is dense in $\H$ and there exists a unique closed symmetric operator $T_X$, defined on $D(T_X)$  such that
$$ \ip{T_X\xi}{\eta} =\int_\RN \lambda  \,d\ip{F(\lambda)\xi}{\eta}, \quad \forall \xi \in D(T_X),\, \eta \in
\H .$$

\end{thm}
\berem We remark that one can prove that
$D(T_X)$ is dense
in $\H$ directly. Indeed, if $\xi\in \H$, we put
$\xi_n=(X(n)-X(-n))\xi$, $n \in {\mb N}$. From ({\sf qs}$_3$) it follows that $\|\xi-\xi_n\|\to 0$. It remains to prove that
$\xi_n\in D(T_X)$.
Taking into account the properties given in Lemma \ref{lemma_f}, one has
\begin{eqnarray*} \int_\RN \lambda ^2
\,d\ip{F(\lambda)\xi_n}{\xi_n}&=& \int_\RN \lambda
^2 \,d\ip{F(\lambda)(X(n)-X(-n))\xi}{(X(n)-X(-n))\xi}\\
&=&\int_\RN \lambda^2\,d\ip{(X(n)^*-X(-n)^*)F(\lambda)(X(n)-X(-n))\xi}{\xi}\\
&=& \int_{-n} ^{n}\lambda^2 d\|X(\lambda)-X(-n))\xi\|^2 \\
&\leq& 2 \int_{-n} ^{n}\lambda^2 d\|X(\lambda)\xi\|^2 \\
&\leq& 2n^2\int_\RN\,d\|X(\lambda)\xi\|^2=2 n^2\|\xi\|^2<\infty.
\end{eqnarray*}
\enrem

Of course, since we have supposed a full symmetry of $\{X(\lambda)\}$ and $\{X(\lambda)^*\}$, we can also define $F_*(\lambda):= X(\lambda)X(\lambda)^*$, $\lambda \in \RN$ and apply the previous statements to the family $\{F_*(\lambda)\}$. Then, by Theorem \ref{thm_opertospectfam}, one defines a second closed symmetric operator $T_{X^*}$ related to the resolution of the identity $\{X(\lambda)\}$.
But there is more.

Assume, in fact, that the resolution of the identity $\{X(\lambda)\}$ is {\em of bounded variation}, by which we mean that, for every $\xi \in \H$ the complex valued function $\lambda \in \RN \to \ip{X(\lambda)\xi}{\xi}$ is of bounded variation on the line. Then there exists a complex Borel measure $\mu_\xi$ on $\RN$ such that
$$\ip{X(\lambda)\xi}{\xi}= \mu_\xi ((-\infty, \lambda)), \quad \lambda \in \RN.$$
By the elementary properties of measures, it follows that
$$ \mu_\xi ((\lambda, \mu))= \ip{X(\mu)\xi-X(\lambda)\xi}{\xi}, \quad \lambda <\mu.$$

It is clear that if  $\{X(\lambda)\}$ is of bounded variation so it is also $\{ X(\lambda)^*\}$, and the corresponding measure is nothing but the complex conjugate of $\mu_\xi$.

\begin{prop}\label{prop_214}Let $\{X(\lambda)\}$ be a *-resolution of the identity. Then the function $\lambda \to X(\lambda)$ is of bounded variation. \end{prop}
\begin{proof}Let $\lambda_0, \ldots, \lambda_n$ be a finite set of points with $-\infty <\lambda_0<\lambda_1 <\cdots \lambda_n<\lambda$. We shorten $\omega=]-\infty, \lambda]$, $\omega_k= ]\lambda_k-\lambda_{k-1}]$. Then we have
\begin{eqnarray*}
\sum_{k=1}^n |\ip{X(\omega_k)\xi}{\eta}|&=&\sum_{k=1}^n|\ip{X(\omega_k)^2\xi}{\eta}|\\
&=&\sum_{k=1}^n |\ip{X(\omega_k)\xi}{X(\omega_k)^*\eta}|\\
&\leq &\sum_{k=1}^n\|X(\omega_k)\xi\|\, \|X(\omega_k)^*\eta\|\\
&=& \sum_{k=1}^n\ip{F(\omega_k)\xi}{\xi}^{1/2}\, \ip{F_*(\omega_k)\eta}{\eta}^{1/2}\\
&\leq &\left(\sum_{k=1}^n \ip{F(\omega_k)\xi}{\xi}\right)^{1/2}\, \left(\sum_{k=1}^n\ip{F_*(\omega_k)\eta}{\eta}\right)^{1/2}\\
&=& \ip{F(\omega)\xi}{\xi}^{1/2}\ip{F_*(\omega)\eta}{\eta}^{1/2}.
\end{eqnarray*}
Hence the supremum over all possible decompositions of $\omega$ is finite and so is the limit when $\lambda \to +\infty$.
\end{proof}

In this case, the non-self-adjoint resolution of the identity $\{X(\lambda)\}$ defines a countably additive measure  $X(\omega)$, on the Borel sets $\omega$ of the line. By the quoted result of Mackey
\cite[Theorem 55]{mackey} one gets the following statement.

\begin{lemma}\label{prop_215}Let $\{X(\lambda)\}$ be a *-resolution of the identity. Then there exists a self-adjoint resolution of the identity $\{E(\lambda)\}$ and an invertible operator $T$ with bounded inverse such that
\begin{equation}\label{mackey_thm} X(\lambda) = T E(\lambda) T^{-1}, \quad \forall \lambda \in {\mb R}.\end{equation}
\end{lemma}

\medskip
Let $\{ X(\lambda)\}$ be a *-resolution of the identity. By Proposition \ref{prop_214} we can define the integral
$$ \int_{\mb R} \lambda d X(\lambda).$$

Let $\D$ denote the set of all $\xi\in \H$ such that the integral $\int_\RN \lambda d\ip{X(\lambda)\xi}{\eta}$ exists, for every $\eta \in \H$, and
the conjugate linear functional
$$ \Theta_\xi (\eta):= \int_\RN \lambda d\ip{X(\lambda)\xi}{\eta}, \quad \eta \in \H$$ is bounded on $\H$.

Then, by the Riesz theorem there exists an operator
$L$ on $\D$ such that
$$\ip{L\xi} {\eta} = \int_\RN \lambda d\ip{X(\lambda)\xi}{\eta}, \quad \eta \in \H.$$
Here we denote the operator $L$ by the integral $ \int_{\mb R} \lambda d X(\lambda).$
We investigate the integral $ \int_{\mb R} \lambda d X(\lambda).$

\begin{lemma}\label{lemma_domain} Let $\{X(\lambda)\}$ be a *-resolution of the identity and let $\{\Delta_k\}$ be a finite or countable family of disjoint intervals of the real line.
Then, for every $\eta \in \H$,
$$\left( \sum_{k=1}^n \|X(\Delta_k)\eta\|^2 \right)^{1/2}\leq \|\eta\|. $$
\end{lemma}
\begin{proof} Since $\{X(\lambda)\}$ is a resolution of the identity,  $X(\lambda)^*X(\lambda)$ is a generalized resolution of the identity. Then we can write $X(\lambda)^*X(\lambda)= \mathbf{PE}(\lambda)\upharpoonright \H$ where, as before, $\mathbf{E}(\cdot)$ is an ordinary (i.e. self-adjoint) resolution of the identity in a possibly larger Hilbert space $\bm{\H}$ and $\mathbf{P}$ the projection onto $\H$. Then we have
\begin{align*}
\sum_{k=1}^n \|X(\Delta_k)\eta\|^2  &= \sum_{k=1}^n \ip{X(\Delta_k)\eta}{X(\Delta_k)\eta}
= \sum_{k=1}^n \ip{X(\Delta_k)^*X(\Delta_k)\eta}{\eta}\\
&= \sum_{k=1}^n\ip{\bm{PE}(\Delta_k)\eta}{\eta}
= \sum_{k=1}^n\ip{\bm{PE}(\Delta_k)\eta}{\eta}\\
&= \sum_{k=1}^n\ip{\bm{E}(\Delta_k)\eta}{\bm{P}\eta}
= \sum_{k=1}^n\ip{\bm{E}(\Delta_k)\eta}{\eta}\\
&= \sum_{k=1}^n\ip{\bm{E}(\Delta_k)\eta}{\bm{E}(\Delta_k)\eta}
= \sum_{k=1}^n \|\bm{E}(\Delta_k)\eta\|^2 \leq \|\eta\|^2.
\end{align*}
\end{proof}

\begin{lemma}\label{lemma_218} We have $D(T_X) = \D$.
\end{lemma}

\begin{proof}For every $\xi \in D(T_X)$, $\eta\in \H$

\begin{equation}\label{ineq_integrals}\left|\int_\RN \lambda d\ip{X(\lambda)\xi}{\eta}\right|\leq \left(\int_\RN \lambda ^2
\,d\ip{F(\lambda)\xi}{\xi}\right)^{1/2} \|\eta\|.\end{equation}

Indeed, let $[\alpha, \beta]$ be a bounded interval, and $\{\Delta_k; \, k=1, \ldots, n\}$  a family of disjoint intervals whose union is $[\alpha, \beta]$. For every $k$, choose $\lambda_k \in \Delta_k$. Then, for the Cauchy sums defining the integrals we get
\begin{eqnarray*}\left|\sum_{k=1}^n \lambda_k \ip{X(\Delta_k)\xi}{\eta}\right|&\leq&\sum_{k=1}^n |\lambda_k |\ip{X(\Delta_k)\xi}{ X(\Delta_k)^*\eta}|\\ &\leq& \left(\sum_{k=1}^n \lambda_k^2 \|X(\Delta_k)\xi\|^2 \right)^{1/2} \left( \sum_{k=1}^n \|X(\Delta_k)^*\eta\|^2 \right)^{1/2}\\
&\leq &  \left(\sum_{k=1}^n \lambda_k^2 \ip{F(\Delta_k)\xi}{\xi} \right)^{1/2} \|\eta\|.
\end{eqnarray*}
Hence, the inequality \eqref{ineq_integrals} holds on every finite interval and, by taking limits also on the real line, and so $D(T_X)\subseteq \D$.

Conversely, , take an arbitrary $\xi \in \D$. Then we have, taking into account Lemma \ref{lemma_f},
\begin{eqnarray*} \|L\xi\|^2 &=&
\int_\RN \lambda d\int_\RN \mu d\ip{X(\lambda)\xi}{X(\mu)\xi}\\
&=&  \int_\RN \lambda d\int_\RN \mu d\ip{X(\mu)^*X(\lambda)\xi}{\xi} \\
&=& \int_\RN \lambda d\int_{-\infty}^\lambda \mu d\ip{X(\mu)^*X(\lambda)\xi}{\xi} +\int_\RN \lambda d\int_\lambda^{\infty} \mu d\ip{X(\mu)^*X(\lambda)\xi}{\xi}\\
&=& \int_\RN \lambda d\int_{-\infty}^\lambda \mu d\ip{X(\mu)^*X(\lambda)^*X(\lambda)\xi}{\xi} +\int_\RN \lambda d\int_\lambda^{\infty} \mu d\ip{X(\mu)^*X(\mu)X(\lambda)\xi}{\xi}\\
&=& \int_\RN \lambda d\int_{-\infty}^\lambda \mu d\ip{X(\mu)^*F(\lambda)\xi}{\xi} +\int_\RN \lambda d\int_\lambda^{\infty} \mu d\ip{F(\mu)X(\lambda)\xi}{\xi}\\
&=& \int_\RN \lambda d\int_{-\infty}^\lambda \mu d\ip{F(\mu)\xi}{\xi} +\int_\RN \lambda d\int_\lambda^{\infty} \mu d\ip{F(\lambda)\xi}{\xi}\\
&=& \int_\RN\lambda^2 d\ip{F(\lambda)\xi}{\xi} .
\end{eqnarray*}
Hence $\D\subseteq D(T_X)$. Thus, we have $\D=D(T_X)$.
\end{proof}

\begin{thm}\label{thm_3.19}Let $B$ be closed operator in $\H$. The following statements are equivalent.
\begin{itemize}
{\item[(i)] $B$ is similar to a self-adjoint operator $A$, that is, $B= TAT^{-1}$, with an intertwining operator $T$ satisfying $\|TE(\lambda)T^{-1}\|=1$, for every $\lambda \in {\mb R}$, where $\{E(\lambda)\}$ is the spectral resolution of $A$.}
\item[(ii)] There exists a *-resolution of the identity $\{X(\lambda)\}$ such that
$$ B= \int_{\mb R} \lambda d X(\lambda).$$
\end{itemize}
\end{thm}
\begin{proof} (i)$\Rightarrow$(ii): Let $A= \int_{\mb R} \lambda d E(\lambda)$ be the spectral resolution of $A$ and put $X(\lambda)=TE(\lambda)T^{-1}$. Then it is easily shown that $\{X(\lambda)\}$ is a *-resolution of the identity (with $\|X(\lambda)\|=1$, for every $\lambda \in {\mb R}$). We show that $D(B)=TD(A)=D(T_X)$.

Let $\xi\in D(T_X)$. Since
\begin{align*} \int_\RN \lambda^2 d\ip{E(\lambda)T^{-1}\xi}{T^{-1}\xi} &= \int_\RN \lambda^2 d\ip{E(\lambda)T^{-1}\xi}{E(\lambda)T^{-1}\xi}\\
&=\int_\RN \lambda^2 d\ip{T^{-1}X(\lambda)\xi}{T^{-1}X(\lambda)\xi}\\
&\leq \|T^{-1}\|^2 \int_\RN \lambda^2 d\ip{X(\lambda)\xi}{X(\lambda)\xi}\\
&=\|T^{-1}\|^2 \int_\RN \lambda^2 d\ip{F(\lambda)\xi}{\xi}<\infty,
\end{align*}
we have, $T^{-1}\xi\in D(A)$ or, equivalently $\xi\in D(B)$.

On the other hand, let $\xi\in TD(A)$. Then,
\begin{align*}
\int_\RN \lambda^2 d\ip{F(\lambda)\xi}{\xi}
&=\int_\RN \lambda^2 d\ip{TE(\lambda)T^{-1}\xi}{TE(\lambda)T^{-1}\xi}\\
&\leq \|T\|^2 \int_\RN \lambda^2 d\ip{E(\lambda)T^{-1}\xi}{T^{-1}\xi}<\infty,
\end{align*}
and so $\xi\in D(T_X)$. Thus, $D(B)=D(T_X)$ and by Lemma \ref{lemma_218}, $D(B)= D(\int_{\mb R}\lambda dX(\lambda))$. Furthermore, we have
$$ \ip{B\xi}{\eta}= \ip{\int_{\mb R}\lambda dX(\lambda)\xi}{\eta}$$
for every $\xi \in D(B)$ and $\eta \in \H$.

(ii)$\Rightarrow$(i): By Lemma \ref{prop_215}, there exists a self-adjoint resolution of the identity $\{E(\lambda)\}$ and an invertible operator $T$ with bounded inverse such that $X(\lambda)=TE(\lambda)T^{-1}$, for every $\lambda \in {\mb R}$.

Let now $A$ be the self-adjoint operator $A=\int_{\mb R}\lambda dE(\lambda)$. As shown above we have
$D(T_X) = TD(A) $ and furthermore $D(B)= D(\int_{\mb R}\lambda dX(\lambda))$ and

$$ \ip{B\xi}{\eta}= \ip{\int_{\mb R}\lambda dX(\lambda)\xi}{\eta}=\ip{TAT^{-1}\xi}{\eta},$$
for every $\xi \in D(B)$ and $\eta \in \H$.

\end{proof}

\berem
The operator $B$ has real spectrum and empty residual spectrum, since $\sigma(B)=\sigma(A)$ and $\sigma_r(B)=\sigma_r(A)$.
Moreover, $B$ is  a {\em pseudo-hermitian} operator; i.e. $B$ is  a spectral operator of scalar type.
\enrem

\begin{cor} \label{cor_321} Let $B$ be a closed operator with positive spectrum. The following statements are equivalent.
\begin{itemize}
{\item[(i)] $B$ is similar to a positive self-adjoint operator $A$, with an intertwining operator $T$ satisfying $\|TE(\lambda)T^{-1}\|=1$, for every $\lambda \in {\mb R}$, where $\{E(\lambda)\}$ is the spectral resolution of $A$.}

\item[(ii)] There exists a *-resolution of the identity $\{X(\lambda)\}_{\lambda \in {\mb R}^+}$ on ${\mb R}^+:=[0, \infty)$ such that
$$ B=\int_0^\infty \lambda dX(\lambda).$$
\end{itemize}
If one of the equivalent conditions (i) or (ii) holds, then there exists a closed operator $B_2$ with positive spectrum such that $B_2^2=B$.
\end{cor}
\begin{proof} The equivalence of (i) and (ii) follow from Proposition \ref{prop_2.4} and Theorem \ref{thm_3.19}. Suppose that $B=TAT^{-1}$ for some positive self-adjoint operator $A$ and an invertible bounded operator $T$ with bounded inverse. Then, putting $B_2= TA^{1/2}T^{-1}$, $B_2$ is a closed operator with positive spectrum such that $B_2^2=B$.
\end{proof}

In order to go further, we need the following lemma

\begin{lemma}\label{lemma_sqrt}
The function $K \mapsto \sqrt{K}$ is strongly continuous on the set $\M:=\{ K \in {\mc B}(\H)\}; K\geq 0,\, \|K\|\leq M\}$.
 \end{lemma}
 \begin{proof} Let $K_\alpha \stackrel{s}{\to} K$, $K_\alpha, K \in \M$. By the Weierstrass theorem, there is a sequence of polynomials $\{p_n(x)\}$ such that $p_n(x) \to \sqrt{x}$, uniformly on $[0, M]$. This implies that $\|p_n( Z) - \sqrt{Z}\|\to 0$, for every $Z \in \M$.
 Since
\begin{align*}
  \|(\sqrt{K_\alpha}- \sqrt{K})\xi\|&\leq \|(\sqrt{K_\alpha} - p_n(K_\alpha))\xi\|+\|(p_n(K_\alpha)- p_n(K))\xi\|\\
  &+\|(p_n(K) - \sqrt{K})\xi\|
\end{align*}
 Now choose, $n$ large enough to make the first and third term in the right hand side smaller than $\epsilon>0$ and, fixed this $n$, take $\alpha$ big enough to make the second term smaller than $\epsilon$. The latter is possible since the multiplication is jointly strongly continuous on every norm bounded ball of ${\mc B}(\H)$; thus, if $K_\alpha \stackrel{s}{\to} K$ then, for every polynomial $p$, $p(K_\alpha) \stackrel{s}{\to} p(K)$.
 \end{proof}

 By Lemma \ref{lemma_sfunction} and Lemma \ref{lemma_sqrt} we obtain the following result.

\begin{prop}\label{prop_sfunction2} Let $\{X(\lambda)\}$ be a *-resolution of the identity. Then, the operator valued function $\lambda \mapsto \Phi(\lambda)$, where $\Phi(\lambda):= F(\lambda)^{1/2}= (X(\lambda)^*X(\lambda))^{1/2}$, has the following properties:
\begin{itemize}
\item[(\sf fs$_1$)] $\displaystyle \sup_{\lambda \in {\mb R}} \|\Phi(\lambda)\|={1}$;
\item[(\sf fs$_2$)] $\Phi(\lambda)\leq \Phi(\mu)$ if $\lambda < \mu$;
\item[(\sf fs$_3$)] $\displaystyle\lim_{\lambda \to -\infty}\Phi(\lambda)\xi = 0; \quad \lim_{\lambda \to +\infty}\Phi(\lambda)\xi = \xi, \quad \forall \xi \in \H$;
\item[(\sf fs$_4$)] $\displaystyle \lim_{\epsilon \to 0^+} \Phi(\lambda +\epsilon)\xi= \Phi(\lambda)\xi, \; \forall \lambda \in {\mb R};\;\forall \xi \in \H$.
\end{itemize}
\end{prop}

Hence, in analogy to Theorem \ref{thm_opertospectfam}, we have

\begin{thm}\label{thm_opertospectfam2} Let $\Phi(\cdot)= (X(\cdot)^*X(\cdot))^{1/2}$ be the positive operator valued function defined by the *-resolution of the identity $\{X(\lambda)\}$ on the real line.
Set, $$D(S_X)=\left\{\xi \in \H:\, \int_\RN
\lambda ^2 \,d\ip{\Phi(\lambda)\xi}{\xi} < \infty\right\}.$$ Then there exists a unique closed symmetric operator $S_X$, defined on $D(S_X)$  such that
$$ \ip{S_X\xi}{\eta} =\int_\RN \lambda  \,d\ip{\Phi(\lambda)\xi}{\eta}, \quad \forall \xi \in D(S_X),\, \eta \in
\H .$$
\end{thm}

\noindent{\bf Question:} Is there any relationship between the operators $B$, $T_X$ and $S_X$? Clearly, if $\{X(\lambda)\}$ is a self-adjoint resolution of the identity, then the three operators coincide. So far we know that $B$ and $T_X$ have the same domain but we do not know if and how $B$ can be expressed in terms of $T_X$. About the relationship between $B$ or $T_X$ we do not know almost anything, so we leave this question open.

\bigskip
\noindent{\bf Acknowledgements --} The authors thank the referee for pointing out a serious inaccuracy in a previous version of this paper.


\medskip
\begin{thebibliography}{99}
\bibitem{jpa_ct_pipandmetric} J.-P. Antoine and C. Trapani, {\em Partial inner product spaces, metric operators and
generalized hermiticity}, J. Phys. A: Math. Theor. {\bf 46} (2013) 025204 (21pp)
\bibitem{bender} C.M. Bender, A. Fring, U. G\"unther and H. Jones, {\em Quantum physics with non-Hermitian
operators}, J. Phys. A: Math. Theor. 45 (2012) 440301
\bibitem{mackey} G.W. Mackey, {\em Commutative Banach Aalgebras}, Notas de Matematica n. 17, Rio de Janeiro, 1959.
\bibitem{mosta1}A.~Mostafazadeh, {\em Pseudo-Hermitian representation of quantum mechanics}, {Int. J. Geom. Methods Mod. Phys.} {\bf 7} (2010)  1191--1306
\bibitem{riesz}F. Riesz and B. Sz. Nagy, Lecons d'Analyse fonctionelle, Gauthier-Villars, Paris (1972)
\bibitem{Akhiezer}N.I. Akhiezer and I.M. Glazman,  Theory of Linear Operators in Hilbert space, II,  Dover Publ. (1993)

\end{thebibliography}
\end{document}